\documentclass[a4paper]{amsart}
\usepackage[latin1]{inputenc}
\usepackage{amssymb}
\usepackage{amsmath}
\usepackage{mathrsfs}
\usepackage{eufrak}
\usepackage{amsthm}
\usepackage{amsfonts}
\usepackage{textcomp}
\usepackage{graphicx}
\usepackage[pdftex]{color}
\usepackage{paralist}
\usepackage[shortlabels]{enumitem}
\usepackage{hyperref}
\usepackage{comment}
\usepackage{tikz}
\usepackage{tikz-cd}
\usepackage[arrow, matrix, curve]{xy}
\usepackage{thmtools}
\usepackage{cleveref}

\newtheorem*{corollary*}{Corollary}
\newtheorem{theorem}{Theorem}[section]
\newtheorem*{theorem*}{Theorem}

\newtheorem{corollary}[theorem]{Corollary}
\newtheorem{Question M}[theorem]{Question M}
\newtheorem{Conjecture S}[theorem]{Conjecture S}
\newtheorem{Conjecture K}[theorem]{Conjecture K}
\newtheorem{lemma}[theorem]{Lemma}
\newtheorem{proposition}[theorem]{Proposition}
\newtheorem{question}[theorem]{Question}
\newtheorem{problem}[theorem]{Problem}

\newtheorem*{claim*}{Claim}
\newtheorem*{conjecture}{Conjecture}

\theoremstyle{definition}
\newtheorem{definition}[theorem]{Definition}

\newtheorem*{theorem }{Theorem}

\newtheorem{example}[theorem]{Example}

\theoremstyle{remark}

\numberwithin{equation}{theorem}

\makeatletter
\renewcommand*\env@matrix[1][\
arraystretch]{%
  \edef\arraystretch{#1}%
  \hskip -\arraycolsep
  \let\@ifnextchar\new@ifnextchar
  \array{*\c@MaxMatrixCols c}}
\makeatother

\newcommand{\Ext}{\operatorname{Ext}}

\newcommand{\im}{\operatorname{Im}}

\newcommand{\Hom}{\operatorname{Hom}}

\newcommand{\grade}{\operatorname{grade}}

\newcommand{\cov}{\operatorname{cov}}

\newcommand{\Ocal}{\mathcal{O}}

\renewcommand{\ker}{\mathrm{Ker}}

\newcommand{\pdim}{\operatorname{pdim}}
\newcommand{\gldim}{\operatorname{gldim}}

\title{Pure minimal injective resolutions and perfect modules for lattices}

\author[T.~Gottesman]{Tal Gottesman}
\address[T.~Gottesman]{Fakult\"{a}t f\"{u}r Mathematik, Ruhr-Universit\"at Bochum, Universit\"{a}tstr. 150, 44801 Bochum, Germany}
\email{tal.gottesman@rub.de}

\author[V.~Kl\'asz]{Vikt\'oria Kl\'asz$^\star$}%
\address[V.~Kl\'asz]{Mathematical Institute of the University of Bonn, Endenicher Allee 60, 53115 Bonn, Germany}%
\email{klasz@math.uni-bonn.de}%
\thanks{$^\star$Supported by the Deutsche Forschungsgemeinschaft
(DFG, German Research Foundation) under Germany's Excellence Strategy - GZ 2047/1, Projekt-ID
390685813\\}

\author[M.~Kleinau]{Markus Kleinau$^\star$}%
\address[M.~Kleinau]{Mathematical Institute of the University of Bonn, Endenicher Allee 60, 53115 Bonn, Germany}%
\email{mkleinau@math.uni-bonn.de}%

\author[R.~Marczinzik]{Ren\'e Marczinzik}%
\address[R.~Marczinzik]{Mathematical Institute of the University of Bonn, Endenicher Allee 60, 53115 Bonn, Germany}
\email{marczire@math.uni-bonn.de}
\date{\today}

\begin{document}
\begin{abstract}
In a recent article, Iyama and Marczinzik showed that a lattice is distributive if and only if the incidence algebra is Auslander regular, giving a new connection between homological algebra and lattice theory. In this article we study when a distributive lattice has a pure minimal injective coresolution, a notion first introduced and studied in a work of Ajitabh, Smith and Zhang. 
We will see that this problem naturally leads to studying when certain antichain modules are perfect modules. We give a classification of perfect antichain modules under the assumption that their canonical antichain resolution is minimal and use this to give a complete classification in lattice theoretic terms of incidence algebras of distributive lattices with pure minimal injective coresolution.
We use our results to answer a question raised by Ajitabh, Smith and Zhang by showing that there exist Auslander-Gorenstein polynomial identity rings without a pure injective coresolution.

\end{abstract}
\maketitle

\section{Introduction}
A fundamental invariant of a two-sided noetherian ring $R$ is its minimal injective coresolution
$$0 \rightarrow R \rightarrow I^0 \rightarrow I^1 \rightarrow \cdots \rightarrow I^n \rightarrow \cdots $$
If this resolution is finite for $R$ as a left and right $R$-module, then $R$ is called \emph{Iwanaga-Gorenstein}. 
If $R$ is Iwanaga-Gorenstein and the flat dimension of $I^i$ is less than or equal to $i$ for all $i \geq 0$, then $R$ is called \emph{Auslander-Gorenstein}. An Auslander-Gorenstein ring that has additionally finite global dimension is called \emph{Auslander regular}.
By classical work of Bass \cite{B}, in the case of commutative local $R$, the Auslander-Gorenstein rings (resp. Auslander regular rings) coincide with the classical Gorenstein rings (resp. regular rings) studied in commutative algebra and algebraic geometry. Many other important classes of noetherian rings are Auslander regular, such as universal enveloping algebra of finite dimensional Lie algebras, Weyl algebras and rings of $\mathbb{C}$-linear differential operators on an irreducible smooth subvariety of an affine space, see for example \cite{VO}. Recently, it was shown that several important classes of finite dimensional algebras are Auslander regular, such as blocks of category $\mathcal{O}$ for semisimple Lie algebras \cite{KMM} and incidence algebras of distributive lattices \cite{IM}. We refer for example to the survey articles \cite{C} and \cite{KM} and the textbook \cite{VO} for more information on Auslander regular rings.
In \cite{ASZ} and \cite{ASZ2}, the authors introduced (in a slightly greater generality) the following concept, which is the center of our study in this article:
\begin{definition}
Let $R$ be a two-sided noetherian Iwanaga-Gorenstein ring with minimal injective coresolution 
$$0 \rightarrow R \rightarrow I^0 \rightarrow I^1 \rightarrow \cdots \rightarrow I^n \rightarrow 0$$
Then $R$ is said to have a \emph{pure minimal injective coresolution} if every non-zero noetherian submodule $X$ of $I^i$ satisfies $\grade X=i$.
\end{definition}
Here the \emph{grade} of a module $M$ is defined as $\grade M= \inf \{ i \geq 0 \mid \Ext_R^i(M,R) \neq 0 \}$, which is a classical homological dimension studied in commutative and non-commutative ring theory, see for example \cite{BH} and \cite{VO}.
In \cite{ASZ}, the authors studied the concept of pure minimal injective coresolutions for several important classes of two-sided noetherian rings of low dimension, such as universal enveloping algebras of Lie algebras, Sklyanin algebras, and certain Artin-Schelter graded regular rings. All of these classes are, in fact, Auslander regular rings.
In this article, we want to study this concept for the first time for finite dimensional algebras.
Let $A$ be a finite dimensional $K$-algebra over a field $K$ in the following. We assume that modules are finitely generated right $A$-modules unless otherwise stated. Note that for a finite dimensional algebra, the flat dimension of a module coincides with the projective dimension of the module.
Following \cite{Bj}, we call a module $M$ \emph{pure} if every submodule $N$ of $M$ has the same grade as $M$. We call $M$ \emph{perfect} if $\grade M= \pdim M$. The study of perfect modules is a classical topic in commutative algebra, see for example \cite{BH}.
By definition, a finite dimensional Iwanaga-Gorenstein algebra $A$ has a minimal pure injective coresolution if and only if every term $I^i$ is pure of grade $i$.

We give a full classification when Auslander regular algebras have a pure minimal injective coresolution:
\begin{theorem} \label{mainresult} (\Cref{cor::AR+pure-iff-inj+simples-perfect})
Let $A$ be an Auslander regular algebra.
Then the following are equivalent:
\begin{enumerate}
    \item $A$ has a minimal pure injective coresolution.
    \item Every simple $A$-module and every indecomposable injective $A$-module is perfect.
\end{enumerate}
\end{theorem}

Note that for an Auslander regular algebra $A$ all simple modules being perfect is equivalent to $A$ being a \emph{right diagonal Auslander regular algebra} in the sense of Iyama \cite{I}, see \cite[Corollary 3.4]{KMT}. Here being right diagonal means that every indecomposable direct summand $X$ of $I^i$ in the minimal injective coresolution of $A$ has projective dimension equal to $i$.
While in general it is extremely difficult to determine whether a given algebra has a minimal pure injective coresolution, the previous Theorem \ref{mainresult} allows us to decide it at least for Auslander regular algebras by reducing the problem to studying homological properties of simple and indecomposable injective modules.
In \cite{IM}, it was shown that the incidence algebra $KL$ of a finite lattice $L$ is Auslander regular if and only if $L$ is distributive, and that in this case, $KL$ is even right diagonal Auslander regular.
Recall that by Birkhoff's classical representation theorem, every finite distributive lattice $L$ is given as the set of order ideals of a finite poset $P$.
We call a poset $P$ \emph{upward-linear} if for all $x \in P$ we have that $x$ is covered by at most one other element in $P$. An example of an upward-linear poset can be found in \Cref{ex::upward_linear_poset}. The upward-linear posets with n elements are in bijection with unlabeled forests of rooted trees with n nodes, see \url{https://oeis.org/A000081}.
Our second main result gives a classification when the incidence algebra of a distributive lattice $L$ has a minimal pure injective coresolution:
\begin{theorem} \label{secondmainresult}(\Cref{pureAreUpwardLinear})
Let $L$ be a finite distributive lattice. Then the incidence algebra $KL$ has a minimal pure injective coresolution if and only if $L$ is the order ideal lattice of an upward-linear poset.
\end{theorem}

We say that a finite dimensional algebra $A$ has a \emph{two-sided minimal pure injective coresolution} when $A$ and its opposite algebra $A^{op}$ both have a minimal pure injective coresolution.
Recall that a lattice $L$ is a \emph{divisor lattice} if $L$ is isomorphic to the lattice of divisors of a natural number $n$.
The following corollary gives a homological characterisation of divisor lattices:
\begin{corollary} (\Cref{cor::2-sided-pure-inj-res-iff-divisor-lattice})
Let $L$ be a finite distributive lattice.
Then $L$ is a divisor lattice if and only if $KL$ has a two-sided pure minimal injective coresolution.

\end{corollary}

The proof of the second main result, \Cref{secondmainresult}, leads us to the study of when an indecomposable injective module is perfect for an incidence algebra of a distributive lattice. In \cite{IM}, it was noted that indecomposable injective modules over incidence algebras $KL$ for general lattices are so-called \emph{antichain modules} and a general canonical projective resolution for antichain modules was established in \cite{IM}. This naturally leads to the following more general question:
\begin{question}
Let $L$ be a lattice. When is an antichain module $M$ over $KL$ perfect?
\end{question}

We give a complete answer to this question under a minimality assumption on the canonical resolution of an antichain module. We also refer to this resolution as the antichain resolution of the module. We refer to the main text for the combinatorial definition of strong. 
\begin{theorem} (\Cref{lem::strong_iff_resolution_minimal} and \Cref{prop::perfect-iff-Boolean})
Let $L$ be a lattice with antichain module $M$.
The antichain resolution of $M$ is minimal if and only if it is strong.
In this case, $M$ is perfect if and only if it is Boolean.
\end{theorem}

\section{Preliminaries}
In the following, we always assume that our algebras are finite dimensional $K$-algebras over a field $K$. Modules will be finitely generated right modules unless otherwise stated. The functors $\Hom$ and $\Ext$ will refer to those of this category. 
Let $\mathcal{P}$ be a finite poset. Then the \emph{incidence algebra} $K\mathcal{P}$ of $\mathcal{P}$ is defined as the quiver algebra $KQ_{\mathcal{P}}/I$, where $Q_{\mathcal{P}}$ is the Hasse diagram of $\mathcal{P}$ and the relations $I$ are given by all commutativity relations of the form $w_1-w_2$ where $w_1$ and $w_2$ are paths of length $\geq 2$ starting and ending at the same points.
For any $x \leq y\in \mathcal{P}$ we denote $p_x^y$ the equivalence class of paths from $x$ to $y$ in $K\mathcal P$. Denote $S(x)$ the simple module associated to $x$. Its projective cover, denoted  $P(x)$, has as basis the set $\{p_{x}^y | x\leq y\}$. Dually, the injective indecomposable module $I(y) = D(Ae_y)$ has as basis the elements $\{D(p_x^y) | x\leq y\}$ where $D$ denotes the vector space dual and the dual basis element, respectively.
We refer for example to \cite{S,Sta} for more information on incidence algebras and their use in combinatorics and representation theory.
Recall that $\mathcal{P}$ is a \emph{lattice} if for every two elements $x,y \in \mathcal{P}$ there is a unique join denoted by $x \lor y$, which is by definition the smallest element bigger than $x$ and $y$ in $\mathcal{P}$, and dually there is a unique meet $x \land y$. Note that every finite lattice has a unique minimum element, denoted as $m,$ and a unique maximum element, denoted as $M.$
A lattice $L$ is called \emph{distributive} if the distributivity law holds for all $x,y,z \in L: (x \land y) \lor z =(x \lor z) \land (y \lor z)$.
A subset $O$ of a poset $\mathcal{P}$ is called an \emph{order ideal} if for all $y \in O$, we also have $x \in O$ if $x \leq y$.
Every subset $A$ of $\mathcal{P}$ induces an order ideal $O$ generated by $A$ simply by taking $O$ to be all elements $x \in \mathcal{P}$ such that there is an $a \in A$ with $x \leq a$.
An \emph{antichain} $C$ in $\mathcal{P}$ is a subset of $\mathcal{P}$ such that any two elements $x,y \in C$ are not comparable.
Note that order ideals are in bijection with antichains in $\mathcal{P}$, where the bijection is given by associating to an order ideal $O$ the minimal antichain that generates $O$. The set of order ideals forms a lattice in a natural way, where the order is given by set inclusion, the meet is given by the usual intersection, and the join is given by the usual union of sets. Recall that in a lattice $L$ a $y\in L$ is called \textit{join-irreducible}, if $a \vee b=y $ for $a,b\in L$ implies $a=y$ or $b=y$.
A central result in the theory of distributive lattices is Birkhoff's representation theorem (see for example \cite[Chapter 3]{Sta} for this theorem and more information on lattice theory) that gives a correspondence between finite posets and distributive lattices:
\begin{theorem}[Birkhoff's representation theorem]
Any finite distributive lattice is isomorphic to the lattice of order ideals of its poset of join irreducible elements. The correspondence sending a poset to its distributive lattice of order ideals and sending a distributive lattice to its poset of join-irreducible elements is a bijection.
\end{theorem}

We now recall some homological notions for finite dimensional algebras. For an introduction to homological algebra and representation theory of finite dimensional algebras we refer for example to \cite{ARS} and \cite{ASS}.
A finite dimensional algebra $A$ is called \emph{Auslander-Gorenstein} if there is a finite minimal injective coresolution of the regular module 
$$0 \rightarrow A \rightarrow I^0 \rightarrow I^1 \rightarrow \cdots \rightarrow I^n \rightarrow 0$$
such that $\pdim I^i \leq i$ for all $i \geq 0$.
An Auslander-Gorenstein algebra is called \emph{Auslander regular} if it additionally has finite global dimension. Note that all incidence algebras have finite global dimension since they are quiver algebras with an acyclic quiver.
A fundamental connection between order theory and the homological algebra of incidence algebras was established in \cite{IM} with the following theorem:
\begin{theorem} \label{theoremausregiffdistr}
Let $L$ be a finite lattice. Then the incidence algebra $KL$ is Auslander regular if and only if $L$ is distributive.
\end{theorem}
We refer to \cite{IM} for a proof and more information such as a formula for the global dimension of incidence algebras of distributive lattices.

A central tool for the proof of the previous theorem is the explicit description of a minimal projective resolution of indecomposable injective modules and, more generally, a projective resolution of antichain modules.

Let $C=\{x_1,\ldots, x_\ell\}$ be an antichain in the finite lattice $L$, and let $A=KL.$ We associate an $A$-module $N_C$ to $C$ as follows: Let $N_C$ be the right ideal of $A$ generated by the elements $p^m_{x_i},$ i.e.  
$$N_C=\sum_{i=1}^\ell p_{x_i}^mA.$$ Here, $p^m_{x_i}$ is the element of $A$ corresponding to the unique path from $m$ to $x_i$ in $A.$ Note that $N_C$ is a submodule of the unique projective-injective $A$-module $P(m)=e_mA\cong I(M)=D(Ae_M).$ In fact, this gives us a bijection between the antichains of $L$ and the submodules of $P(m),$ see \cite[Proposition 2.1]{IM}. Then the \textit{antichain module} associated to $C$ is defined to be $$M_C=P(m)/N_C.$$ Note that a module is an antichain module if and only if it is a factor module of $P(m).$ Then by \cite[Theorem 2.2]{IM}, we can explicitly describe a projective resolution of the antichain module $M_C$ as follows:

\begin{equation}
    \label{eq::antichain_resolution}
    0\rightarrow P_\ell \xrightarrow{\partial_\ell} \ldots \xrightarrow{\partial_2} P_1 \xrightarrow{\partial_1} P_0 \xrightarrow{\partial_0} M_C \rightarrow 0 \text{ with }P_0=P(m)\text{ and } P_r=\bigoplus_{S\subseteq C, |S|=r}P(\vee S). 
\end{equation}
The boundary maps are defined by choosing an arbitrary ordering of the antichain $C$ and taking a linear combination of canonical maps $P(\vee S) \to P(\vee (S\setminus \{x\}))$ according to the Koszul sign convention. For more details, see \cite[Theorem 2.2]{IM}.
We also refer to (\ref{eq::antichain_resolution}) as the \textit{antichain resolution} of $M_C$. 
In general, this resolution is not minimal. In \Cref{lem::strong_iff_resolution_minimal} we show that the resolution is minimal if and only if $C$ is a so-called \textit{strong} antichain. 

An important special case of antichain modules that we are interested in this paper is indecomposable injective modules. Indeed, every indecomposable injective is a factor module of $P(m)$, and the indecomposable injective $I(x)=M_C$ for the antichain $C=\min ([m,x]^c)$, i.e., the set of minimal elements of $L$ outside the interval $[m,x].$  It was proven in \cite[Proposition 3.1]{IM} that if $L$ is distributive, all elements of  $\min ([m,x]^c)$ are join-irreducible. A main result of \cite{IM} establishes that in case $L$ is distributive, the antichain resolution (\ref{eq::antichain_resolution}) of an indecomposable injective $I(x)$ is always minimal.

\begin{theorem}\cite[Theorem 3.2]{IM}
\label{thm::min_resolution_injectives_distributive}
    Let $L$ be a distributive lattice and $x\in L.$ Then (\ref{eq::antichain_resolution}) defines a minimal projective resolution for $I(x)\cong M_C$ where $C=\min([m,x])^c.$
    
    Moreover, $\pdim I(x)=|\cov(x)|.$
\end{theorem}

There exists a dual statement which provides a minimal injective resolution for every indecomposable projective $KL$-module in the case when $L$ is distributive, see \cite[Theorem 3.2]{IM}.

We will need the following lemma:
\begin{lemma} \label{Bensonlemma}
Let $A$ be a finite dimensional algebra, $N$ an indecomposable $A$-module, and $S$ a simple $A$-module. Let $(P_i)_i$ be a minimal projective resolution of $N$ and $(I^i)_i$ a minimal injective coresolution of $N$. 
\begin{enumerate} 
\item For $l \geq 0$, $\Ext_A^{l}(N,S) \neq 0$ iff $S$ is a quotient of $P_l$.  \item For $l \geq 0$, $\Ext_A^{l}(S,N) \neq 0$ iff $S$ is a submodule of $I^l$. \end{enumerate} 
\end{lemma} 
\begin{proof}
See \cite[Corollary 2.5.4]{Ben}
\end{proof}

\section{Pure modules and algebras}
The next definition is due to \cite[Definition 0.3]{ASZ}, specialised to the situation of finite dimensional algebras. See also \cite{C} and \cite{ASZ2} for more information and examples for noetherian rings.
\begin{definition} \label{maindefinition}
Let $A$ be an Iwanaga-Gorenstein algebra with selfinjective dimension $d$.
An $A$-module $M$ is called \emph{pure} if $\grade(N)=\grade(M)$ for all non-zero submodules $N$ of $M$. We say that $A$ has a \emph{pure minimal injective coresolution} if in the minimal injective coresolution of the regular module 
$$0 \rightarrow A \rightarrow I^0 \rightarrow I^1 \rightarrow \cdots \rightarrow I^d \rightarrow 0$$
each module $I^i$ is pure with $\grade I^i=i$

\end{definition}
It is quite complicated to see directly from the definition whether a given algebra has a pure minimal injective coresolution.
We just give a simple example:
\begin{example}
Let $A=KQ$ be the path algebra of an acyclic quiver $Q$.
Then every indecomposable non-projective $A$-module $M$ satisfies $\Hom_A(M,A)=0$.
First, assume $Q$ is not of linearly oriented Dynkin type $A_n$. Then there is no non-zero projective-injective $A$-module and thus the zeroth term $I^0$ in the minimal injective coresolution 
$$0 \rightarrow A \rightarrow I^0 \rightarrow I^1 \rightarrow 0$$
has no projective direct summands. Thus, $\grade I^0=1$ and $A$ has no minimal pure injective coresolution.
Next, let $Q$ be linearly oriented of Dynkin type $A_n$:
$Q=$
\[\begin{tikzcd}
	1 & 2 & \ldots & {n-1} & n
	\arrow[from=1-1, to=1-2]
	\arrow[from=1-2, to=1-3]
	\arrow[from=1-3, to=1-4]
	\arrow[from=1-4, to=1-5]
\end{tikzcd}\]
Let $I(k)$ denote the indecomposable injective module corresponding to vertex $k$. Note that $I(k)$ is projective if and only if $k=n$.
Then the minimal injective coresolution of $A$ looks as follows:
$$0 \rightarrow A \rightarrow I(n)^n \rightarrow \bigoplus\limits_{k=1}^{n-1}{I(k)} \rightarrow 0.$$
Now every submodule $X$ of $I(n)^n$ satisfies $\Hom_A(X,A) \neq 0$ since $\Hom_A(X,I(n)^n) \neq 0$ and $I(n)$ is a submodule of $A$. Thus every submodule of $I(n)^n$ has grade equal to zero and $I(n)^n$ is pure of grade 0. 
Every submodule $Y$ of $\bigoplus\limits_{k=1}^{n-1}{I(k)}$ has a non-projective injective envelope and thus has no projective direct summands (as all projective modules embed into a power of $I(n)$). Thus every such submodule $Y$ of $\bigoplus\limits_{k=1}^{n-1}{I(k)}$ has grade equal to 1 and $\bigoplus\limits_{k=1}^{n-1}{I(k)}$ is pure of grade 1.
\end{example}

In the rest of this section, we classify when Auslander regular algebras have a pure minimal injective coresolution.
\begin{proposition} \label{lemmapure}
Let $A$ be an Auslander regular algebra and assume that $A$ has a pure minimal injective coresolution. Then:
\begin{enumerate}
    \item $A$ is right diagonal Auslander regular.
    \item All indecomposable injective right $A$-modules are perfect.
\end{enumerate}
\end{proposition}
\begin{proof}
Let 
$$0 \rightarrow A \rightarrow I^0 \rightarrow I^1 \rightarrow \cdots \rightarrow I^n \rightarrow 0$$ be a minimal injective coresolution of $A$.
We first show (1):
 We have to show that every direct summand $X$ of $I^i$ has projective dimension equal to $i$. But $X$ is a submodule of $I^i$ and the pureness assumption directly gives us $\grade(X)=i$. Since $X$ is a direct summand of $I^i$ and since $A$ is Auslander regular, we have  that $\pdim X \leq \pdim I^i \leq i$.
    Thus $i = \grade X \leq \pdim X \leq i$ and therefore $\pdim X=i$ and $X$ is perfect. Now (2) follows since every indecomposable injective right $A$-module is a direct summand of some $I^i$.
To see this, note that by Lemma \ref{Bensonlemma} (2), this is equivalent to the existence of some $i_S \geq 0$ such that $\Ext_A^{i_S}(S,A) \neq 0$ for every simple $A$-module $S$. But for a module $M$ of finite projective dimension $g$, we have in general that $\Ext_A^g(M,A) \neq 0$ and every $A$-module $M$ has finite projective dimension since $A$ is Auslander regular by assumption.
Thus all indecomposable injective right $A$-modules are direct summands of some $I^i$ and thus are perfect.
\end{proof}

We also need the next result, which is a rather deep result about Auslander regular algebras whose proof relies on spectral sequences:
\begin{theorem} \label{Bjoerklemma}
Let $A$ be an Auslander regular algebra and $M$ an $A$-module. Then $M$ has the property that $\grade N= \grade M$ for every non-zero submodule $N$ of $M$ if and only if $\Ext_{A^{op}}^v( \Ext_A^v(M,A),A) =0$ for all $v \neq \grade M$.
\end{theorem}
\begin{proof}
See \cite[Proposition 2.6, Appendix IV]{Bj}.
\end{proof}
We can now prove our first main result:
\begin{theorem} \label{propositioncharacterisepureausreg}
Let $A$ be an Auslander regular algebra.
Then $A$ has a right pure minimal injective coresolution if and only if $A$ is right diagonal and every indecomposable injective right $A$-module is perfect.
\end{theorem}
\begin{proof}
 We already saw that if $A$  has a right pure minimal injective coresolution then it is right diagonal and every indecomposable injective right $A$-module is perfect in Proposition \ref{lemmapure}.
 Now assume that $A$ is right diagonal and every indecomposable injective right $A$-module is perfect. Then $\grade I^i =i$ for all $i$ since $\pdim I^i = i$ by the right diagonal Auslander regular condition and all since all indecomposable direct summands of $I^i$ are perfect modules. Now $I^i$ has the property that all non-zero submodules have the same grade as $I^i$ if and only if $\Ext_{A^{op}}^v( \Ext_A^v(I^i,A),A) =0$ for all $v \neq \grade I^i=i$ by Theorem \ref{Bjoerklemma}.
 But $\Ext$ is additive with respect to direct summands and so this condition simplifies to $\Ext_{A^{op}}^v( \Ext_A^v(L,A),A) =0$ for all $v \neq i$ for all indecomposable direct summands $L$ of $I^i$. 
 Such an indecomposable injective direct summand $L$ is perfect of grade $i$ and thus $\Ext_A^v(L,A)=0$ for $v \neq i$, which proves the claim. 
\end{proof}

\begin{corollary}\label{cor::AR+pure-iff-inj+simples-perfect}
Let $A$ be a finite dimensional algebra of finite global dimension. Then the following are equivalent:
\begin{enumerate}
    \item $A$ is an Auslander regular algebra with a pure minimal injective coresolution.
    \item Every indecomposable injective $A$-module and every simple $A$-module is perfect, that is: For every indecomposable injective $A$-module, we have $\pdim I=\grade I=\grade S=\pdim S$.
\end{enumerate}

\end{corollary}
\begin{proof}
This follows from Lemma \ref{lemmapure} and Proposition \ref{propositioncharacterisepureausreg} together with \cite[Theorem 1.1 and Corollary 3.4]{KMT}, which state that being Auslander regular is equivalent to $\grade S=\pdim I(S)$ for every simple module $S$ with injective envelope $I(S)$ and that being right diagonal for an Auslander regular algebra is equivalent to all simple modules being perfect.
\end{proof}

\begin{corollary} \label{purereducedtoinjperfectcor}
The incidence algebra $A$ of a distributive lattice is pure if and only if every indecomposable injective right $A$-module is perfect.
\end{corollary}
\begin{proof}
 This follows directly since incidence algebras of distributive lattices are right diagonal and using Theorem \ref{propositioncharacterisepureausreg}.
\end{proof}

We remark that there are incidence algebras $KP$ of finite posets $P$ that are not lattices that are Auslander regular with a pure minimal injective coresolution as the next example shows:
\begin{example} \label{nonlatticeexample}
Consider the poset $P$ below.
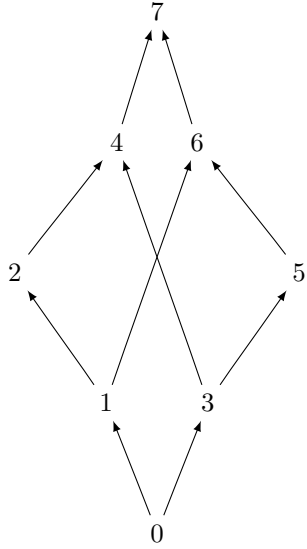
\[\begin{tikzpicture}[>=latex,line join=bevel,]
\node (node_0) at (59.0bp,6.5bp) [draw,draw=none] {$0$};
  \node (node_1) at (40.0bp,55.5bp) [draw,draw=none] {$1$};
  \node (node_3) at (78.0bp,55.5bp) [draw,draw=none] {$3$};
  \node (node_2) at (6.0bp,104.5bp) [draw,draw=none] {$2$};
  \node (node_6) at (74.0bp,153.5bp) [draw,draw=none] {$6$};
  \node (node_4) at (44.0bp,153.5bp) [draw,draw=none] {$4$};
  \node (node_5) at (112.0bp,104.5bp) [draw,draw=none] {$5$};
  \node (node_7) at (59.0bp,202.5bp) [draw,draw=none] {$7$};
  \draw [black,->] (node_0) ..controls (54.05bp,19.746bp) and (49.688bp,30.534bp)  .. (node_1);
  \draw [black,->] (node_0) ..controls (63.95bp,19.746bp) and (68.312bp,30.534bp)  .. (node_3);
  \draw [black,->] (node_1) ..controls (30.935bp,69.031bp) and (22.645bp,80.491bp)  .. (node_2);
  \draw [black,->] (node_1) ..controls (47.187bp,76.792bp) and (60.97bp,115.71bp)  .. (node_6);
  \draw [black,->] (node_2) ..controls (16.247bp,118.17bp) and (25.789bp,129.98bp)  .. (node_4);
  \draw [black,->] (node_3) ..controls (70.813bp,76.792bp) and (57.03bp,115.71bp)  .. (node_4);
  \draw [black,->] (node_3) ..controls (87.065bp,69.031bp) and (95.355bp,80.491bp)  .. (node_5);
  \draw [black,->] (node_4) ..controls (47.908bp,166.75bp) and (51.351bp,177.53bp)  .. (node_7);
  \draw [black,->] (node_5) ..controls (101.75bp,118.17bp) and (92.211bp,129.98bp)  .. (node_6);
  \draw [black,->] (node_6) ..controls (70.092bp,166.75bp) and (66.649bp,177.53bp)  .. (node_7);
\end{tikzpicture}\]

Then $P$ is not a lattice but $KP$ is diagonally Auslander regular with all indecomposable injective modules perfect. Thus it has a pure minimal injective coresolution. 
We found this example using \cite{QPA} and refer to the example \ref{qpacodecomplicatedexample} in the appendix for a computer software verification.
\end{example}

We thus pose the following problem:

\begin{problem}
Classify the posets $P$ such that $KP$ has a pure minimal injective coresolution.
\end{problem}

Based on several hundred examples, we expect that an algebra of finite global dimension with a minimal pure injective coresolution is in fact Auslander regular. We pose this as a question: 
\begin{question}
Let $A$ be an algebra of finite global dimension and a minimal pure injective coresolution. Is $A$ Auslander regular?
\end{question}

When $A=KL$ is the incidence algebra of a lattice we pose this as a conjecture. We verified this conjecture for all finite lattices with at most 10 elements:
\begin{conjecture}
Let $A=KL$ be the incidence algebra of a finite lattice $L$ and assume that $A$ has a pure minimal injective coresolution.
Then $L$ is distributive.
\end{conjecture}
We will see in a later section the complete classification of distributive lattices $L$ with $KL$ having a pure minimal injective coresolution using our main result, Theorem \ref{propositioncharacterisepureausreg} of this section.

We end this section with a simple example on how to show that the incidence algebra $KL$ of a non-distributive lattice $L$ does not have a minimal pure injective coresolution. The method to look at indecomposable injective direct summands of $I^i$ usually, but not always, works:
\begin{example}
Let $L$ be the pentagon lattice:
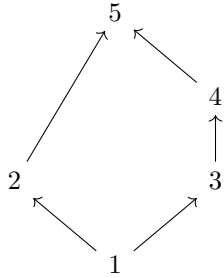
\[\begin{tikzcd}
	& 5 \\
	&& 4 \\
	2 && 3 \\
	& 1
	\arrow[from=2-3, to=1-2]
	\arrow[from=3-1, to=1-2]
	\arrow[from=3-3, to=2-3]
	\arrow[from=4-2, to=3-1]
	\arrow[from=4-2, to=3-3]
\end{tikzcd}\]

Then $A=KL$ has the minimal injective coresolution
$$0 \rightarrow A \rightarrow I^0 \rightarrow I^1 \rightarrow I^2 \rightarrow 0$$
with $\grade I^i=i$, but 
$I^1$ is not pure, since the indecomposable injective module $I(3)$ is a submodule of $I^1$ with grade equal to 2.
\end{example}

\section{On perfect antichain modules in lattices}

In this section, we consider a lattice $L$ which is not necessarily distributive. We will use terminology from \cite{Go} only adapting them to the convetions of this paper. We say that an antichain $C$ of $L$ is \emph{strong} if for all subsets $S, S'$ of $C$, it holds that if $\vee S \leq \vee S'$ then $S \subseteq S'$. We say that an antichain is \emph{Boolean} if any two subsets $S, S'$ of $C$ satisfy the following identity
\[\vee S \wedge \vee S' = \vee (S\cap S').\]
The goal of this section is to give an algebraic interpretation for modules associated to Boolean antichains. 

\begin{lemma}
\label{lem::strong_iff_resolution_minimal}
Let $L$ be a lattice and $C$ an antichain in $L$. Then $C$ is strong if and only if the antichain resolution associated to $C$ is minimal.
\end{lemma}

\begin{proof}
The proof of the direct statement can be found in \cite[Proposition 3.2.4]{Go2}. We reproduce it here for completeness.
Fix an integer $k\leq |C|$. To see that the kernel of $\ker(\partial_k) = \im(\partial_{k +1})$ is superfluous in
\[\bigoplus_{|S| = k} P(\vee S)\] 
recall that the image of $\partial_{k+1}$ is generated by paths associated to relations $\vee S \geq \vee (S-\{i\})$ of the posets. Because the antichain is strong we always have $\vee S > \vee (S-\{i\})$ so that the lazy paths are not in $\ker(\partial_k)$. 

Conversely, suppose $C$ is not a strong antichain. Then there exists $S, S'$ such that $\vee S\leq \vee S'$ and $S\not\subseteq S'$. According to \cite[Lemma 2.1.1]{Go}, we can assume that $S$ and $S'$ have the same cardinality. We denote it by $k$. Let $x$ be an element of $S$ but not of $S'$. By construction we have 
\[\vee(S'\cup \{x\}) = x\vee (\vee S') = \vee S'.\]
Hence the projective indecomposable $P(\vee S')$ appears in degree $k+1$ as well. The image of the map $\partial_k$ is thus not superfluous and the antichain resolution is not minimal.
\end{proof}

The following lemma reformulates the notion of a Boolean antichain in a way which is more suitable for discussing the properties of their antichain resolution.
\begin{lemma}\label{lemma:LocalBoolean}An antichain $C$ of a lattice $L$ is Boolean \emph{iff} for all $S\subset C$ such that $|S| \leq |C| - 2$ we have
\begin{equation}\label{eq:BooleanReformulation}\vee S = \bigwedge_{x\in C\setminus S} (\vee (S \cup \{x\})).\end{equation}
\end{lemma}

\begin{proof}
If $C$ is a Boolean antichain, the equality above holds. Conversely, assume that these equalities hold for all subsets of $C$ of cardinality at most $|C|-2$. Let $S$ and $S'$ be arbitrary subsets of $C$. It always holds that
\[(\vee S) \wedge (\vee S')  \geq \vee(S\cap S').\]
We will prove that every subset $E$ of $C$ which contains $S\cap S'$ satisfies
\begin{equation}(\vee S) \wedge (\vee S') \leq \vee E\end{equation}
by reverse induction on the size of $E$. First notice that the inequality is clear when $C = E$. When $|E| = |C| - 1$, if $S \not\subset E$ then $S' \subset E$ and the inequality holds again. Assume that it holds for any subset containing $S\cap S'$ of cardinality $k$ for some fixed $k \leq |C| - 2$. Consider a subset $E$ of cardinality $k-1$. By \cref{eq:BooleanReformulation} we have 
\[\vee E = \bigwedge_{x\in E^c} (\vee (E \cup \{x\})).\]
By the induction hypothesis we then have that for every $x\in E^c$, 
\[\vee (E \cup \{x\}) \geq (\vee S)\wedge (\vee S').\]
This concludes the induction step. The minimal case, when $E = S\cap S'$ ensures that $C$ is a Boolean antichain.
\end{proof}

\begin{proposition}\label{prop::perfect-iff-Boolean}
Let $L$ be a lattice and $C$ a strong antichain in $L$. Then the antichain module $M$ associated to $C$ is perfect if and only if $C$ is Boolean.
\end{proposition}
\begin{proof}
    Suppose that the antichain $C$ is Boolean. Let $x$ be an element of the lattice. Note that the projective indecomposable module $P(x)$ is the interval module associated to the interval $[x, M]$. By \cite[Theorem 2.2.3]{Go}, the extension group $\Ext^i(M, P(x))$ is non zero if and only if the set $E = \{S \subseteq C| \vee S \geq x\}$ is a singleton $\{S_x\}$ and $|S_x| = i$. However, if $E$ is not empty, $E$ automatically contains $C$. From this observation we conclude that the $\grade M = |C|$ and that $M$ is a perfect module.
    Conversely suppose that the antichain $C$ is not Boolean. According to \Cref{lemma:LocalBoolean} there exists a subset $S$ of $C$ of cardinality $r = |S| \leq |C| - 2$ such that
    \[\vee S < \bigwedge_{x\in C\setminus S} (\vee (S \cup \{x\})).\]
    We denote $y = \bigwedge_{x\in C\setminus S} (\vee (S \cup \{x\}))$. We claim that $\Ext^{r+1}(M, P(y))$ is non-zero. The group $\Ext^{r+1}(M, P(y))$ is the $r^{th}$ homology group of the total Hom functor
    \begin{multline*}
    \Hom(P(m), P(y)) \to\dots
    \to\Hom\Big(\bigoplus_{\substack{S'\subseteq C\\|S'| = r}} P(\vee S'), P(y)\Big) \to
    \Hom\Big(\bigoplus_{\substack{S'\subseteq C\\|S'| = r+1}} P(\vee S'), P(y)\Big) \to \\ \dots\to \Hom(P(\vee C), P(y)).
    \end{multline*}
    Recall that for any two elements $a, b$ of $L$
    \[\Hom(P(a) , P(b)) \cong \begin{cases}K &\text{if } a \geq b \\ 0 & \text{otherwise},\end{cases}\]
    and that $C$ was assumed to be strong. It follows that the total $\Hom$ complex is isomorphic to the following complex of vector spaces
    \[0 \to \dots\to 0\to \bigoplus_{\substack{S\subset S' \subset C\\ |S'| = r+1}} K \to\dots\to \bigoplus_{\substack{S\subset S' \subset C\\ |S'| = i}} K \to  \dots \to K.\]
    This is the the simplicial complex of the simplicial set
    \[\{S' \subseteq C | S \subset S'\}.\]
    Being a so-called \emph{stupid truncation} of the simplex $\{S' \subseteq C | S \subseteq S' \}$, this chain complex has indeed homology in degree $r+1 < |C|$. Because the antichain $C$ is strong, the projective dimension of $M$ is $|C|$. Hence, we have shown that $M$ is not perfect.
\end{proof}

By \cite[Corollary 3.3]{IM}, we know that the antichain resolution of injective indecomposable modules in a distributive lattice is minimal. Hence their corresponding antichain is strong and we have the following corollary.
\begin{corollary}
    Let $
L$ be a distributive lattice. Then an indecomposable injective module is perfect \emph{iff} its antichain is Boolean.
\end{corollary}

\section{Classification of distributive lattices with perfect injective modules}
We saw in \Cref{purereducedtoinjperfectcor} that the classification of distributive lattices with a pure minimal injective coresolution is reduced to the classification when all indecomposable injective modules are perfect.
In this section we will give this classification and thus also classify all distributive lattices with a pure minimal injective coresolution.\\
We start by giving a combinatorial description of the grade of an interval module.



\begin{proposition}\cite[Proposition 2.3, Appendix IV]{Bj} 
\label{prop::grade_SES}
    Let $0\rightarrow M\rightarrow E \rightarrow N\rightarrow 0$ be a short exact sequence over an Auslander-Gorenstein algebra. Then 
    \[\grade(E) = \min \{\grade(M),\grade(N)\}\]
\end{proposition}

If $a\leq b$ are two elements of the poset $L$, we denote the interval module associated to the interval $[a,b]$ by $M(a,b).$ More precisely, $M(a,b)$ is the unique indecomposable $KL$-module whose  composition factors are simple modules of the form $S(c)$ for some $a\le c\le b$, and $M(a,b)$ contains each such $S(c)$ exactly once. 
\begin{proposition}\label{gradeFormula}
    Let $L$ be a distributive lattice and $a\leq b\in L$. Then 
    \[\grade(M(a,b)) = \min_{x\in [a,b]}(|\cov(x)|)\]
\end{proposition}
\begin{proof}
By \Cref{prop::grade_SES}, $\grade(M(a,b))$ equals the minimum of the grades of its composition factors. Thus, 
\[\grade(M(a,b)) = \min_{x\in [a,b]}(\grade S(x) ),\] where $S(x)$ is the simple $KL$-module corresponding to vertex $x.$
Since we have $\grade S(x)=|\cov(x)|$ in a distributive lattice by \cite[Corollary 3.3]{IM}, the statement follows. 

\end{proof}
To decide whether an injective module is perfect we can compare the description of the grade to the description of the projective dimension given in \cite[Theorem 3.2]{IM}, see also in \Cref{thm::min_resolution_injectives_distributive}.\\
Recall that a poset $P$ is called \emph{upward-linear} if every $x \in P$ has at most one cover. Here is an example of an upward linear poset $P$ with its distributive lattice of order ideals $L$:
\begin{example}
\label{ex::upward_linear_poset}
See \cref{figure:upwardLinearAndLattice} for an upward linear poset $P$ and its lattice of order ideals $\Ocal(P)$.

\begin{figure}[h]
\begin{minipage}{0.3\textwidth}\centering
\begin{tikzpicture}[>=latex,line join=bevel,]
\node (node_0) at (6.0bp,6.5bp) [draw,draw=none] {$5$};
  \node (node_1) at (36.0bp,6.5bp) [draw,draw=none] {$2$};
  \node (node_2) at (36.0bp,55.5bp) [draw,draw=none] {$3$};
  \node (node_4) at (51.0bp,104.5bp) [draw,draw=none] {$4$};
  \node (node_3) at (66.0bp,55.5bp) [draw,draw=none] {$1$};
  \draw [black,->] (node_1) ..controls (36.0bp,19.603bp) and (36.0bp,30.062bp)  .. (node_2);
  \draw [black,->] (node_2) ..controls (39.908bp,68.746bp) and (43.351bp,79.534bp)  .. (node_4);
  \draw [black,->] (node_3) ..controls (62.092bp,68.746bp) and (58.649bp,79.534bp)  .. (node_4);
\end{tikzpicture} 
\end{minipage}
\begin{minipage}{0.6\textwidth}\centering
\begin{tikzcd}[ampersand replacement=\&,column sep = {6 em,between origins},row sep = {4 em,between origins}]
	\&\&\& {\{1,2,3,4,5\}}\& \\
	\&\& {\{1,2,3,5\}} \&\& {\{1,2,3,4\}}\& \\
	\& {\{1,2,5\}} \& {\{2,3,5\}} \& {\{1,2,3\}}\& \\
	{\{1,5\}} \& {\{2,5\}} \& {\{1,2\}} \& {\{2,3\}}\& \\
	{\{5\}} \& {\{1\}} \& {\{2\}}\& \\
	\& {\{\}}\&
	\arrow[from=2-3, to=1-4]
	\arrow[from=2-5, to=1-4]
	\arrow[from=3-2, to=2-3]
	\arrow[from=3-3, to=2-3]
	\arrow[from=3-4, to=2-3]
	\arrow[from=3-4, to=2-5]
	\arrow[from=4-1, to=3-2]
	\arrow[from=4-2, to=3-2]
	\arrow[from=4-2, to=3-3]
	\arrow[from=4-3, to=3-2]
	\arrow[from=4-3, to=3-4]
	\arrow[from=4-4, to=3-3]
	\arrow[from=4-4, to=3-4]
	\arrow[from=5-1, to=4-1]
	\arrow[from=5-1, to=4-2]
	\arrow[from=5-2, to=4-1]
	\arrow[from=5-2, to=4-3]
	\arrow[from=5-3, to=4-2]
	\arrow[from=5-3, to=4-3]
	\arrow[from=5-3, to=4-4]
	\arrow[from=6-2, to=5-1]
	\arrow[from=6-2, to=5-2]
	\arrow[from=6-2, to=5-3]
\end{tikzcd}
\end{minipage}
\caption{From left to right, $P$ and $\Ocal(P)$}
\label{figure:upwardLinearAndLattice}
\end{figure}

\end{example}

The Hasse quivers of upward-linear posets are forests of rooted trees. The forests of rooted trees on $n$ vertices are in bijection with the rooted trees on $n+1$ vertices by adding a new vertex that covers all roots in the forests. In particular, the number of upward-linear posets up to isomorphisms is given by \url{https://oeis.org/A000081}.

\begin{theorem}\label{pureAreUpwardLinear}
Let $L$ be a distributive lattice. Then the following are equivalent:
\begin{enumerate}
    \item $KL$ is Auslander regular with pure minimal injective coresolution.
    \item All indecomposable injective $KL$ modules are perfect.
    \item $L$ is isomorphic to the lattice of order ideals of an upward-linear poset.
    \item For every cover relation $a\lessdot b$ in $L$ we have $|\cov(a)|\geq |\cov(b)|$.
\end{enumerate}
\end{theorem}
\begin{proof}
    We have already seen $(1) \Leftrightarrow (2)$ in \Cref{purereducedtoinjperfectcor}.\\
    $(2)\Rightarrow (4)$: Let $a\lessdot b$ be a cover relation in $L$. We consider the injective $I(b)$. Because $I(b)$ is perfect, we have $\grade(I(b)) = \pdim(I(b)) = |\cov(b)|$. Here, the second equality follows from \Cref{thm::min_resolution_injectives_distributive}. From \Cref{gradeFormula} and $I(b) = M(m,b)$ we get for all $c\leq b$ that $|\cov(c)|\geq \grade(I(b)) =|\cov(b)|$. In particular $|\cov(a)|\geq |\cov(b)|$.\\
    $(4)\Rightarrow(2)$: By induction condition $(4)$ implies that $|\cov(a)|\geq |\cov(b)|$ for all $a\leq b$ in $L$. Now fix an $a\in L$ and consider the injective $I(a)$. We have
    \[\grade(I(a)) = \min_{x\in [m,a]}(|\cov(x)|) =|\cov(a)| = \pdim(I(a))\]
    where the first equality is \Cref{gradeFormula}, the second follows from induction on condition $(4)$ and third is a consequence of \Cref{thm::min_resolution_injectives_distributive}.\\
    $(3)\Rightarrow(4)$: Let $\mathcal{P}$ be the poset of join-irreducibles of $L$. For an element $x\in L$ we write $O(x)$ for the corresponding order ideal in $\mathcal P$. We fix $a\lessdot b\in L$ and let $\{x_1,\dots,x_k\}$ be the minimal elements of $\mathcal P \setminus O(a)$. Then the elements covering $a$ in $L$ are given as order ideals by the sets $O(a)\cup \{x_i\}$ for $1\leq i \leq k$. In particular $|\cov(a)| = k$. Without loss of generality let $O(b) = O(a)\cup \{x_k\}$. The minimal elements of $\mathcal P\setminus O(b)$ are the elements $\{x_1,\dots x_{k-1}\}$ and all covers of $x_k$ in $\mathcal P$ that are not comparable to any of the other $x_i$. Because $\mathcal P$ is upward-linear, $x_k$ has at most one cover and we get $|\cov(b)|\leq (k-1) +1 = k = |\cov(a)|$.\\
    $(4)\Rightarrow (3)$: Now assume that $\mathcal P$ is not upward-linear. Then there exists a $y\in \mathcal P$ that admits two different covers $x_1$ and $x_2$. Let $a\in L$ be given by $O(a) = \{z\in \mathcal P \mid y\nleq z\}$. Then $y$ is the only minimal element in $\mathcal P\setminus O(a)$ and hence $|\cov(a)| = 1$. Let $b$ be given by $O(b) = O(a)\cup\{y\}$. Then $b$ covers $a$ in $L$. In addition $x_1$ and $x_2$ are two different minimal elements in $\mathcal P \setminus O(b)$. In particular $|\cov(a)|= 1 < 2\leq|\cov(b)|$.  
\end{proof}

The theorem also yields a characterization of the two-sided condition.
\begin{corollary}\label{cor::2-sided-pure-inj-res-iff-divisor-lattice}
        Let $L$ be a distributive lattice. Then the following are equivalent:
        \begin{enumerate}
            \item $KL$ has a two-sided pure minimal injective coresolution.
            \item $L$ is a divisor lattice.
        \end{enumerate}
\end{corollary}
\begin{proof}
    Let $\mathcal P$ be the poset of join-irreducibles of $L$. By \Cref{pureAreUpwardLinear} the algebra $KL$ has a two-sided minimal injective coresolution if and only if $P$ and $P^{op}$ are upward-linear. That is every element of $P$ has at most one cover and at most one cocover. This is true if and only if $P$ is a disjoint union of chains. The divisor lattices are exactly the lattices of order ideals of disjoint unions of chains.
\end{proof}

We end this section with the following problem:
\begin{problem}
Let $L$ be a finite lattice with incidence algebra $KL$.
Classify $L$ such that all simple and indecomposable injective left and right modules are perfect.
\end{problem}
Here is an example of a lattice satisfying this condition, which is not distributive:
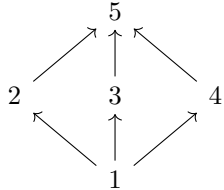
\[\begin{tikzcd}
	& 5 \\
	2 & 3 & 4 \\
	& 1
	\arrow[from=2-1, to=1-2]
	\arrow[from=2-2, to=1-2]
	\arrow[from=2-3, to=1-2]
	\arrow[from=3-2, to=2-1]
	\arrow[from=3-2, to=2-2]
	\arrow[from=3-2, to=2-3]
\end{tikzcd}\]
In forthcoming work the previous problem will be studied for geometric lattices.
As a final application of our results we answer the following question raised by Ajitabh, Smith and Zhang in the beginning of section 4 in \cite{ASZ2}:
\begin{question} \label{ASZquestion}
Let $A$ be an Auslander-Gorenstein ring that is also a polynomial identity ring. Does $A$ have a pure minimal injective coresolution?

\end{question}
Note here that in \cite{ASZ2} the authors just talk about a pure injective coresolution, but for finite dimensional algebras a pure injective coresolution exists if and only if there is a pure minimal injective coresolution.
Here we recall that a ring $R$ is called \emph{polynomial identity ring}  if there is, for some $N > 0$, a polynomial  $F \neq 0$ in the free algebra, $\mathbb{Z}\langle x_1,x_2,\ldots,x_N \rangle$, over the ring of integers in N variables $x_1, x_2, \ldots, x_N$ such that 
$F(r_1,r_2,\ldots,r_N)=0$ for all $N$-tuples $r_1,\ldots,r_N$ of elements in $R$. 
Now by the following result of Feinberg every incidence algebra of a finite poset is a polynomial identity ring:
\begin{theorem} \cite[Theorem 2]{Fe} \label{Feinbergtheorem}
Let $P$ be a finite poset with incidence algebra $A=KP$ and let $F(x_1,\ldots,x_k):=\sum\limits_{\pi \in S_k}^{}{\operatorname{sign} \pi \cdot  x_{\pi(1)} \cdot\ldots\cdot x_{\pi(k)}}$ be the standard polynomial in $k$-variables, where $S_k$ denotes the symmetric group in $k$ letters and $\operatorname{sign} \pi$ the signum of the permutation $\pi$.
Then $A$ is a polynomial identity algebra with 
$F(r_1,\ldots,r_{2n})=0$ for all $r_1,\ldots,r_{2n} \in A$, where $n$ is such that all chains of the poset $P$ have length $\leq n-1$.
\end{theorem}

Now using our main result \Cref{pureAreUpwardLinear}, we can give a distributive lattice $L$ with incidence algebra $A$ such that $A$ is Auslander-Gorenstein and a polynomial identity ring (by Theorem \ref{Feinbergtheorem}) but such that $A$ does not have a pure injective coresolution.
The simplest example of such an $A$, where we do not have a pure injective coresolution of $A$ and $A^{op}$ is given in the following example:
\begin{example}
Let $L$ be the distributive lattice with Hasse diagram 
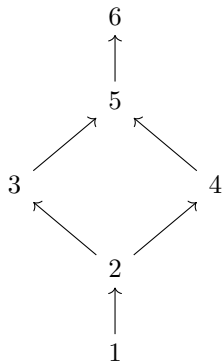
\[\begin{tikzcd}
	& 6 \\
	& 5 \\
	3 && 4 \\
	& 2 \\
	& 1
	\arrow[from=2-2, to=1-2]
	\arrow[from=3-1, to=2-2]
	\arrow[from=3-3, to=2-2]
	\arrow[from=4-2, to=3-1]
	\arrow[from=4-2, to=3-3]
	\arrow[from=5-2, to=4-2]
\end{tikzcd}\]

Let $A$ be the incidence algebra of $L$. Then obviously $A \cong A^{op}$ and by Theorem \ref{pureAreUpwardLinear} $A$ and $A^{op}$ do not have a pure injective coresolution.
But as $L$ is distributive, $A$ is Auslander-Gorenstein and by Theorem \ref{Feinbergtheorem} $A$ is a polynomial identity ring, answering Question \ref{ASZquestion} in the negative.
\end{example}

\section{Appendix: Verifying the pure minimal injective coresolution condition for Auslander regular algebras using QPA}

In this appendix we show how to use the GAP-package QPA to test whether a given Auslander regular algebra $A=KQ/I$ with finite connected acyclic quiver $Q$ and admissible ideal $I$ has a pure minimal injective coresolution. We assume here that $Q$ is acyclic so that we have an explicit bound on the global dimension of $A$, namely $\gldim A \leq n-1$ where $n$ is the number of points in $Q$, see for example \cite{MS}. By our main theorem \ref{propositioncharacterisepureausreg}, this is equivalent to $A$ being Auslander regular with all simple and indecomposable injective $A$-modules perfect.
Here is the QPA code that you can copy and paste into your terminal in order to use the command "IsAuslanderregularwithPMIC" that will check whether a given $KQ/I$ as above is indeed Auslander regular with a minimal pure injective coresolution:

\begin{tiny}
\begin{verbatim}
DeclareOperation("gradeofmodule",[IsList]);
InstallMethod(gradeofmodule, "for a representation of a quiver", [IsList],0,function(LIST)
local A,M,RegA,g,temp,i,UT,nn;
A:=LIST[1];
M:=LIST[2];
nn:=Size(SimpleModules(A));
if IsProjectiveModule(M)=true then return(0);else
RegA:=DirectSumOfQPAModules(IndecProjectiveModules(A));
g:=GlobalDimensionOfAlgebra(A,nn);
temp:=[];Append(temp,[Size(HomOverAlgebra(M,RegA))]);
for i in [0..g-1] do Append(temp,[Size(ExtOverAlgebra(NthSyzygy(M,i),RegA)[2])]);od;
UT:=Filtered([0..g],x->temp[x+1]>0);
return(Minimum(UT));fi;
end);

DeclareOperation("kGortestdegree",[IsList]);
InstallMethod(kGortestdegree, "for a representation of a quiver", [IsList],0,function(LIST)
local A,g,injA,CoRegA,temp,temp2,i,temp3,uu,nn;
A:=LIST[1];
nn:=Size(SimpleModules(A));
g:=GlobalDimensionOfAlgebra(A,nn);
injA:=IndecInjectiveModules(A);CoRegA:=DirectSumOfQPAModules(injA);
temp:=[];for i in [0..g] do Append(temp,[Source(ProjectiveCover(NthSyzygy(CoRegA,i)))]);od;
temp2:=[];for i in [0..g] do Append(temp2,[i-InjDimensionOfModule(Source(ProjectiveCover(NthSyzygy(CoRegA,i))),nn)]);od;
temp3:=Filtered([0..g],x->temp2[x+1]<0);
Append(temp3,[g]);
uu:=Minimum(temp3);
return(uu);
end);

DeclareOperation("IsAuslanderGorenstein",[IsList]);
InstallMethod(IsAuslanderGorenstein, "for a representation of a quiver", [IsList],0,function(LIST)
local A,g,gg,nn;
A:=LIST[1];
nn:=Size(SimpleModules(A));
g:=GorensteinDimensionOfAlgebra(A,nn);
gg:=kGortestdegree([A]);
return(g-gg=0);
end);

DeclareOperation("isperfectmodule",[IsList]);
InstallMethod(isperfectmodule, "for a representation of a quiver", [IsList],0,function(LIST)
local A,M,t,tt,nn;
A:=LIST[1];
M:=LIST[2];
nn:=Size(SimpleModules(A));
if IsProjectiveModule(M)=true then return(1=1);else
t:=ProjDimensionOfModule(M,nn);
tt:=gradeofmodule([A,M]);
return(t=tt);fi;
end);

DeclareOperation("allsimperfect",[IsList]);
InstallMethod(allsimperfect, "for a representation of a quiver", [IsList],0,function(LIST)
local A,simA,W;
A:=LIST[1];
simA:=SimpleModules(A);
W:=Filtered(simA,x->isperfectmodule([A,x])=true);
return(W=simA);
end);

DeclareOperation("allinjperfect",[IsList]);
InstallMethod(allinjperfect, "for a representation of a quiver", [IsList],0,function(LIST)
local A,injA,W;
A:=LIST[1];
injA:=IndecInjectiveModules(A);
W:=Filtered(injA,x->isperfectmodule([A,x])=true);
return(W=injA);
end);

DeclareOperation("IsAuslanderregularwithPMIC",[IsList]);
InstallMethod(IsAuslanderregularwithPMIC, "for a representation of a quiver", [IsList],0,function(LIST)
local A;
A:=LIST[1];
return(IsAuslanderGorenstein([A]) and allsimperfect([A]) and allinjperfect([A]));
end);
\end{verbatim}
\end{tiny}

Here is a finite dimensional quiver algebra that has a pure minimal injective coresolution that is not an incidence algebra:
\begin{example}
Let $A=KQ/I$ with
$Q=$
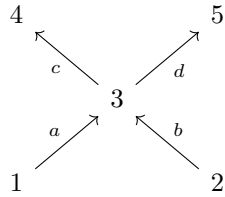
\[\begin{tikzcd}
	4 && 5 \\
	& 3 \\
	1 && 2
	\arrow["c", from=2-2, to=1-1]
	\arrow["d"', from=2-2, to=1-3]
	\arrow["a", from=3-1, to=2-2]
	\arrow["b"', from=3-3, to=2-2]
\end{tikzcd}\]
and $I=\langle ac,bd \rangle$.

    The algebra is gentle and Auslander regular with a pure minimal injective coresolution. For more information on the classifcation of gentle Auslander-Gorenstein algebras, see \cite{Kla}.
    Here is a QPA code on how to verify that $A$ is indeed Auslander regular with minimal pure injective coresolution:
    \begin{tiny}
    \begin{verbatim}
        Q:=Quiver(5,[[1,3,"a"],[2,3,"b"],[3,4,"c"],[3,5,"d"]]);
        KQ:=PathAlgebra(Rationals,Q);AssignGeneratorVariables(KQ);rel:=[a*c,b*d];
        A:=KQ/rel;
        IsAuslanderregularwithPMIC([A]);
    \end{verbatim}
    \end{tiny}
\end{example}

We next verify that the example from \Cref{nonlatticeexample} is indeed Auslander regular with a minimal pure injective coresolution:
\begin{example} \label{qpacodecomplicatedexample}
To obtain the incidence algebra $KP$ in QPA and verify that it is indeed Auslander regular with a minimal pure injective coresolution use the following code:
\begin{tiny}
\begin{verbatim}
W:=[["'x0'", "'x1'", "'x2'", "'x3'", "'x4'", "'x5'", "'x6'", "'x7'"],
[["'x0'", "'x1'"], ["'x0'", "'x3'"], ["'x1'", "'x2'"], ["'x1'", "'x6'"], ["'x2'", "'x4'"], ["'x3'", "'x4'"], 
["'x3'", "'x5'"], ["'x4'", "'x7'"], ["'x5'", "'x6'"], ["'x6'", "'x7'"]], [["'x0'", "'x1'"], ["'x0'", "'x2'"], 
["'x0'", "'x3'"], ["'x0'", "'x4'"], ["'x0'", "'x5'"], ["'x0'", "'x6'"], ["'x0'", "'x7'"], ["'x1'", "'x2'"], 
["'x1'", "'x4'"], ["'x1'", "'x6'"], ["'x1'", "'x7'"], ["'x2'", "'x4'"], ["'x2'", "'x7'"], ["'x3'", "'x4'"], 
["'x3'", "'x5'"], ["'x3'", "'x6'"], ["'x3'", "'x7'"], ["'x4'", "'x7'"], ["'x5'", "'x6'"], ["'x5'", "'x7'"],
["'x6'", "'x7'"]]];P:=Poset(W[1],W[2]);A:=PosetAlgebra(Rationals,P);IsAuslanderregularwithPMIC([A]);
\end{verbatim}
\end{tiny}
\end{example}


\begin{thebibliography}{MTY}






\bibitem[ARS]{ARS} Auslander, M.; Reiten, I.; Smalo, S. O.: {\it Representation theory of Artin algebras.} Cambridge Studies in Advanced Mathematics, Volume 36, Cambridge University Press, 1997.  



\bibitem[ASS]{ASS} Assem, I.; Simson, D.; Skowronski, A.: {\it Elements of the representation theory of associative algebras. Vol. 1: Techniques of representation theory. } London Mathematical Society Student Texts 65. Cambridge: Cambridge University Press, 458 p. (2006). 

\bibitem[ASZ]{ASZ} Ajitabh, K.; Smith, S. P.; Zhang, J. J.: {\it Injective resolutions of some regular rings.} J. Pure Appl. Algebra 140, No. 1, 1-21 (1999). 

\bibitem[ASZ2]{ASZ2} Ajitabh, K.; Smith, S. P.; Zhang, J. J.: {\it Auslander-Gorenstein rings. } Commun. Algebra 26, No. 7, 2159-2180 (1998).


\bibitem[B]{B} Bass, H.: {\it On the ubiquity of Gorenstein rings.} Math. Z. 82, 8-28 (1963).


\bibitem[Ben]{Ben} Benson, D.: {\it Representations and cohomology I: Basic representation theory of finite groups and associative algebras.} Cambridge Studies in Advanced Mathematics, Volume 30, Cambridge University Press, 1991.

\bibitem[Bj]{Bj} Bj\"ork, J.: {\it Analytic D-modules and applications. }  Mathematics and its Applications (Dordrecht). 247. Dordrecht: Kluwer Academic Publishers. xiii, 581 p. (1993). 


\bibitem[BH]{BH} Bruns, W.; Herzog, J.: {\it Cohen-Macaulay rings. Rev. ed.} Cambridge Studies in Advanced Mathematics. 39. Cambridge: Cambridge University Press. (1998). 




\bibitem[C]{C} Clark, J. : {\it Auslander--Gorenstein Rings for Beginners.} International Symposium on Ring Theory , 2001, 95-115.


 
\bibitem[Fe]{Fe} Feinberg, R.: {\it Polynomial identities of incidence algebras.} Proc. Am. Math. Soc. 55, 25-28 (1976). 



\bibitem[Go]{Go} Gottesman, T.:{\it Fractionally Calabi-Yau lattices that tilt to higher Auslander algebras of type A.} https://arxiv.org/abs/2406.09148

\bibitem[Go2]{Go2} Gottesman, T.:{\it Ensemble ordonn\'es Calabi-Yau fractionaire} (PhD thesis)




\bibitem[I]{I} Iyama, O.: {\it The relationship between homological properties and representation theoretic realization of artin algebras.} 	Trans. Amer. Math. Soc. 357 (2005), no. 2, 709-734.

\bibitem[IM]{IM} Iyama, O.; Marczinzik, R.: {\it Distributive lattices and Auslander regular algebras.} Advances in Mathematics
Volume 398, 2022.

\bibitem[Kl{\'a}]{Kla} Kl\'asz, V.: {\it The Auslander-Gorenstein condition for monomial algebras.} https://arxiv.org/pdf/2508.06957.


\bibitem[KMT]{KMT} Kl\'asz, V.; Marczinzik, R.; Thomas, H.: {\it Auslander regular algebras and Coxeter matrices.} https://arxiv.org/abs/2501.09447

\bibitem[KM]{KM} Kl\'asz, V.; Marczinzik, R.: {\it A survey on Auslander-Gorenstein algebras.} https://arxiv.org/abs/2508.19079.

\bibitem[KMM]{KMM} Ko, H.; Mazorchuk, V.; Mrden, R.: {\it Some homological properties of category $\mathcal{O}$. VI.} Doc. Math. 26, 1237-1269 (2021).

\bibitem[MS]{MS} Marczinzik, R.; Sen, E.: {\it A new characterization of quasi-hereditary Nakayama algebras and applications.} Commun. Algebra 50, No. 10, 4481-4493 (2022).








\bibitem[QPA]{QPA} The QPA-team, QPA - Quivers, path algebras and representations - a GAP package, Version 1.33; 2022 https:
//folk.ntnu.no/oyvinso/QPA/. 



\bibitem[S]{S} Simson, D.:{\it  Linear representations of partially ordered sets and vector space categories.} Gordon and Breach Science
Publishers, Algebra, Logic and Applications, Volume 4, 1992.

\bibitem[Sta]{Sta} Stanley, R.: {\it Enumerative combinatorics. Vol. 1. 2nd ed. }  Cambridge Studies in Advanced Mathematics 49.(2012).






\bibitem[VO]{VO} Van Oystaeyen, F.: {\it Algebraic geometry for associative algebras.} Pure and Applied Mathematics 232. New York, NY: Marcel Dekker. vi, 287 p. (2000).


\end{thebibliography}
\end{document}